\documentclass[reqno]{amsart}

\usepackage{amssymb}
\usepackage{graphicx}
\usepackage{amscd}
\usepackage[hidelinks]{hyperref}
\usepackage{color}
\usepackage{float}
\usepackage{graphics,amsmath,amssymb}
\usepackage{amsthm}
\usepackage{amsfonts}
\usepackage{latexsym}
\usepackage{epsf}
\usepackage{xifthen}
\usepackage{mathrsfs}
\usepackage{dsfont}
\usepackage{makecell}
\usepackage[FIGTOPCAP]{subfigure}
\usepackage{amsmath}
\allowdisplaybreaks[4]
\usepackage{listings}
\usepackage{etoolbox}
\usepackage{fancyhdr}
\usepackage{pdflscape}
\usepackage[title,toc,titletoc]{appendix}
\usepackage{enumitem}
\usepackage[noadjust]{cite}
\usepackage{forest}

 \newtheoremstyle{mytheorem}
 {3pt}
 {3pt}
 {\slshape}
 {}
 {\bfseries}
 {.}
 { }
 {}

\numberwithin{equation}{section}

\theoremstyle{theorem}
\newtheorem{theorem}{Theorem}[section]
\newtheorem*{theorem*}{Theorem}

\newtheorem{lemma}[theorem]{Lemma}

\newtheorem{innercustomgeneric}{\customgenericname}
\providecommand{\customgenericname}{}
\newcommand{\newcustomtheorem}[2]{%
	\newenvironment{#1}[1]
	{%
		\renewcommand\customgenericname{#2}%
		\renewcommand\theinnercustomgeneric{##1}%
		\innercustomgeneric
	}
	{\endinnercustomgeneric}
}
\newcustomtheorem{ctheorem}{Theorem}
\newcustomtheorem{clemma}{Lemma}

\theoremstyle{definition}
\newtheorem{definition}{Definition}[section]

\newtheorem{example}{Example}[section]
\newtheorem*{example*}{Example}
\newtheorem*{examples*}{Examples}

\newtheorem*{remark*}{Remark}
\newtheorem*{remarks*}{Remarks}

\newtheoremstyle{named}{}{}{\itshape}{}{\bfseries}{.}{.5em}{#1\thmnote{ #3}}
\theoremstyle{named}

\newcommand{\Keywords}[1]{\ifthenelse{\isempty{#1}}{}{\smallskip \smallskip \noindent \textbf{Keywords}. #1}}
\newcommand{\MSC}[2][2010]{\ifthenelse{\isempty{#2}}{}{\smallskip \smallskip \noindent \textbf{#1MSC}. #2}}
\newcommand{\abstractnote}[1]{\ifthenelse{\isempty{#1}}{}{\smallskip \smallskip \noindent \textsuperscript{\dag}#1}}

\makeatletter
\def\specialsection{\@startsection{section}{1}%
  \z@{\linespacing\@plus\linespacing}{.5\linespacing}%
  {\normalfont}}
\def\section{\@startsection{section}{1}%
  \z@{.7\linespacing\@plus\linespacing}{.5\linespacing}%
  {\normalfont\scshape}}
\patchcmd{\@settitle}{\uppercasenonmath\@title}{\Large\boldmath}{}{}
\patchcmd{\@settitle}{\begin{center}}{\begin{flushleft}}{}{}
\patchcmd{\@settitle}{\end{center}}{\end{flushleft}}{}{}
\patchcmd{\@setauthors}{\MakeUppercase}{\normalsize}{}{}
\patchcmd{\@setauthors}{\centering}{\raggedright}{}{}
\patchcmd{\section}{\scshape}{\large\bfseries\boldmath}{}{}
\patchcmd{\subsection}{\bfseries}{\bfseries\boldmath}{}{}
\renewcommand{\@secnumfont}{\bfseries}
\patchcmd{\@startsection}{\@afterindenttrue}{\@afterindentfalse}{}{}
\patchcmd{\abstract}{\leftmargin3pc}{\leftmargin1pc}{}{}

\def\maketitle{\par
  \@topnum\z@ 
  \@setcopyright
  \thispagestyle{empty}
  \ifx\@empty\shortauthors \let\shortauthors\shorttitle
  \else \andify\shortauthors
  \fi
  \@maketitle@hook
  \begingroup
  \@maketitle
  \toks@\@xp{\shortauthors}\@temptokena\@xp{\shorttitle}%
  \toks4{\def\\{ \ignorespaces}}
  \edef\@tempa{%
    \@nx\markboth{\the\toks4
      \@nx\MakeUppercase{\the\toks@}}{\the\@temptokena}}%
  \@tempa
  \endgroup
  \c@footnote\z@
  \@cleartopmattertags
}
\makeatother

\newcommand{\bB}{\mathbf{B}}

\newcommand{\cB}{\mathcal{B}}
\newcommand{\cL}{\mathcal{L}}

\newcommand{\mA}{\mathscr{A}}

\newcommand{\mI}{\mathscr{I}}

\newcommand{\bA}{\mathbf{A}}
\newcommand{\bW}{\mathbf{W}}

\newcommand{\sfO}{\mathsf{O}}

\newcommand{\LHS}{\operatorname{LHS}}
\newcommand{\RHS}{\operatorname{RHS}}

\newcommand{\ubA}{\underline{\mathbf{A}}}
\newcommand{\uba}{\underline{\boldsymbol{\alpha}}}
\newcommand{\ubb}{\underline{\boldsymbol{\beta}}}
\newcommand{\ubgg}[1]{\underline{\boldsymbol{\gamma_{#1}}}}

\newcommand{\Mat}{\operatorname{Mat}}
\newcommand{\diag}{\operatorname{diag}}

\title{Linked partition ideals and a family of quadruple summations}

\author[G. E. Andrews]{George E. Andrews}
\address[G. E. Andrews]{Department of Mathematics, The Pennsylvania State University, University Park, PA 16802, USA}
\email{gea1@psu.edu}

\author[S. Chern]{Shane Chern}
\address[S. Chern]{Department of Mathematics and Statistics, Dalhousie University, Halifax, Nova Scotia, B3H 4R2, Canada}
\email{chenxiaohang92@gmail.com}

\date{}

\begin{document}

\maketitle

\begin{abstract}

Recently, $4$-regular partitions into distinct parts are connected with a family of overpartitions. In this paper, we provide a uniform extension of two relations due to Andrews for the two types of partitions. Such an extension is made possible with recourse to a new trivariate Rogers--Ramanujan type identity, which concerns a family of quadruple summations appearing as generating functions for the aforementioned overpartitions. More interestingly, the derivation of this Rogers--Ramanujan type identity is relevant to a certain well-poised basic hypergeometric series.


\MSC{11P84, 05A17.}
\end{abstract}

\section{Introduction}

In the theory of basic hypergeometric series and integer partitions, the two Rogers--Ramanujan identities play an irreplaceable role. From an analytic perspective, they are
\begin{align}
	\prod_{n\ge 0}\frac{1}{(1-q^{5n+1})(1-q^{5n+4})} &= \sum_{n\ge 0}\frac{q^{n^2}}{(q;q)_n},\label{eq:RR1}\\
	\prod_{n\ge 0}\frac{1}{(1-q^{5n+2})(1-q^{5n+3})} &= \sum_{n\ge 0}\frac{q^{n^2+n}}{(q;q)_n}.\label{eq:RR2}
\end{align}
Here and throughout we adopt the $q$-Pochhammer symbols for $n\in\mathbb{N}\cup\{\infty\}$,
\begin{align*}
	(A;q)_n:=\prod_{k=0}^{n-1} (1-A q^k)
\end{align*}
and
\begin{align*}
	(A_1,\ldots,A_\ell;q)_n := (A_1;q)_n\cdots (A_\ell;q)_n.
\end{align*}
In terms of integer partitions, the two identities may be interpreted as follows.
\begin{ctheorem}{RR}
	\textup{(i)}~The number of partitions of $n$ into parts congruent to $\pm 1$ modulo $5$ is the same as the number of partitions of $n$ such that every two consecutive parts have difference at least $2$.
	
	\textup{(ii)}~The number of partitions of $n$ into parts congruent to $\pm 2$ modulo $5$ is the same as the number of partitions of $n$ such that every two consecutive parts have difference at least $2$ and that the smallest part is greater than $1$.
\end{ctheorem}

Since the first discovery of \eqref{eq:RR1} and \eqref{eq:RR2} by Rogers \cite{Rog1894}, which were overlooked for nearly two decades until Ramanujan \cite{Ram1914,Ram1919} and Schur \cite{Sch1917} independently reproduced them, there have been numerous generalizations and analogs of the Rogers--Ramanujan identities, among which Gordon's extension \cite{Gor1961} to higher moduli is of substantial significance. Subsequently, Andrews \cite{And1974} established the analytic counterpart of Gordon's result, namely, for $1\le i\le k$ and $k\ge 2$,
\begin{equation}
	\prod_{\substack{n\ge 1\\n\not\equiv 0,\pm i\, (\operatorname{mod}\, 2k+1)}}\frac{1}{1-q^n}=\sum_{n_1,\ldots,n_{k-1}\ge 0}\frac{q^{N_1^2+N_2^2+\cdots+ N_{k-1}^2+N_i+N_{i+1}+\cdots +N_{k-1}}}{(q;q)_{n_1}(q;q)_{n_2}\cdots (q;q)_{n_{k-1}}},
\end{equation}
where $N_j=n_j+n_{j+1}+\cdots +n_{k-1}$. Summarizing from the right-hand side of the above, we are led to a family of $q$-multi-summations now known as the \emph{series of Andrews--Gordon type}:
\begin{equation}\label{eq:And-conj-0}
	\sum_{n_1,\ldots,n_r\ge 0}\frac{(-1)^{L_1(n_1,\ldots,n_r)}q^{Q(n_1,\ldots,n_r)+L_2(n_1,\ldots,n_r)}}{(q^{A_1};q^{A_1})_{n_1}\cdots (q^{A_r};q^{A_r})_{n_r}},
\end{equation}
in which $L_1$ and $L_2$ are linear forms and $Q$ is a quadratic form in the indices $n_1,\ldots,n_r$. It is usually expected to construct Andrews--Gordon type series or summations of alike shapes so that they are equal to a certain infinite product. Along this line, our first object is the following trivariate relation.

\begin{theorem}\label{th:quadruple-ind}
	We have
	\begin{align}\label{eq:quadruple-ind}
		&(-xq;q^2)_\infty (-yq^2;q^4)_\infty\notag\\
		&\quad= \sum_{n_1,n_2,n_3,n_4\ge 0}\frac{x^{n_1+n_2+2n_4}y^{n_2+n_3}q^{n_1+3n_2+2n_3+4n_4}(1+x^2yq^{6+8(n_1+n_2+n_3+n_4)})}{(q^2;q^2)_{n_1}(q^2;q^2)_{n_2}(q^4;q^4)_{n_3}(q^4;q^4)_{n_4}}\notag\\
		&\quad\quad\times q^{4\binom{n_1}{2}+6\binom{n_2}{2}+4\binom{n_3}{2}+8\binom{n_4}{2}+4n_1n_2+4n_1n_3+4n_1n_4+4n_2n_3+4n_2n_4+4n_3n_4}.
	\end{align}
\end{theorem}

This identity is mainly motivated by a recent work of Andrews \cite{And2022} on $4$-regular partitions into distinct parts; here a partition is \emph{$k$-regular} if no part is divisible by $k$. Owing to a theorem of Glaisher \cite{Gla1883}, such partitions are also equinumerous with partitions with no part appearing $k$ or more times and this definition is often used in representation theory \cite[p.~251]{JK1981}. In \cite{And2022}, Andrews connected $4$-regular partitions into distinct parts with overpartitions that were introduced by Corteel and Lovejoy \cite{CL2004}. Recall that an \emph{overpartition} of $n$ is a partition of $n$ where the first occurrence of each distinct part may be overlined. For example, $4$ has fourteen overpartitions:
\begin{gather*}
	4,\ \overline{4},\ 3+1,\ \overline{3}+1,\ 3+\overline{1},\ \overline{3}+\overline{1},\ 2+2,\ \overline{2}+2,\\
	2+1+1,\ \overline{2}+1+1,\ 2+\overline{1}+1,\ \overline{2}+\overline{1}+1,\ 1+1+1+1,\ \overline{1}+1+1+1.
\end{gather*}
Now consider the set $\mA_{\{\overline{1}\}}^\veebar$ of overpartitions such that
\begin{enumerate}[label={\textup{(\arabic*)~}},leftmargin=*,labelsep=0cm,align=left]
	\item Only odd parts {\bfseries\boldmath larger than $1$} may be overlined;
	
	\item The difference between any two parts is $\ge 4$ and the inequality is strict if the larger one is overlined or divisible by $4$ with {\bfseries\boldmath the exception that $\overline{5}$ and $1$ may simultaneously appear as parts}.
\end{enumerate}
Andrews proved the following two results.

\begin{ctheorem}{A1}\label{th:A1}
	Let $A_1(n,m)$ count the number of overpartitions of $n$ in $\mA_{\{\overline{1}\}}^\veebar$ into $m$ parts with overlined parts and parts divisible by $4$ counted with weight $2$. Further, let $B_1(n,m)$ count the number of partitions into $m$ distinct parts none divisible by $4$. Then
	\begin{align*}
		A_1(n,m)=B_1(n,m).
	\end{align*}
\end{ctheorem}

\begin{ctheorem}{A2}\label{th:A2}
	Let $A_2(n,m)$ count the number of overpartitions of $n$ in $\mA_{\{\overline{1}\}}^\veebar$ into $m$ parts with overlined parts counted with weight $3$ and even parts
	counted with weight $2$. Further, let $B_2(n,m)$ count the number of partitions into $m$ odd parts none appearing more than three times. Then
	\begin{align*}
		A_2(n,m)=B_2(n,m).
	\end{align*}
\end{ctheorem}

Now note that
\begin{align*}
	\sum_{m,n\ge 0}B_1(n,m)x^m q^n &= \prod_{\substack{k\ge 1\\k\not\equiv 0 \, (\operatorname{mod}\, 4)}} (1+xq^{k})\\
	& = (-xq;q^2)_\infty (-xq^2;q^4)_\infty.
\end{align*}
Meanwhile,
\begin{align*}
	\sum_{m,n\ge 0}B_2(n,m)x^m q^n &= \prod_{k\ge 1} (1+xq^{2k-1}+x^2q^{2(2k-1)}+x^3q^{3(2k-1)})\\
	& = (-xq;q^2)_\infty (-x^2q^2;q^4)_\infty.
\end{align*}
Hence the two infinite products are special cases of the left-hand side of \eqref{eq:quadruple-ind}. Naturally, it is then expected that the right-hand side of \eqref{eq:quadruple-ind} should characterize the overpartition set $\mA_{\{\overline{1}\}}^\veebar$.

For this purpose, we first loosen the conditions for $\mA_{\{\overline{1}\}}^\veebar$.

\begin{definition}
	Let $\mA$ denote the set of overpartitions such that
	\begin{enumerate}[label={\textup{(\arabic*)~}},leftmargin=*,labelsep=0cm,align=left]
		\item Only odd parts may be overlined;
		
		\item The difference between any two parts is $\ge 4$ and the inequality is strict if the larger one is overlined or divisible by $4$.
	\end{enumerate}
\end{definition}

Our next object is to establish quinvariate generating function formulas for the above overpartitions, possibly with extra restrictions on the smallest part, such as
$$\sum_{\lambda\in \mA} x^{\sharp(\lambda)}y_1^{\sharp_{2,4}(\lambda)}y_2^{\sharp_{0,4}(\lambda)}z^{\sfO(\lambda)}q^{|\lambda|}.$$
Here we adopt the notations that for any (over)partition $\lambda$, $|\lambda|$ and $\sharp(\lambda)$ are the sum of all parts (namely, the \emph{size}) and the number of parts (namely, the \emph{length}) in $\lambda$, respectively, and $\sharp_{a,M}(\lambda)$ is the number of parts in $\lambda$ that are congruent to $a$ modulo $M$. Meanwhile, we denote by $\sfO(\lambda)$ the number of overlined parts in an overpartition $\lambda$.

For the sake of brevity, we postpone the presentation of these generating functions until Theorem \ref{th:quin-gf}. However, we state here that Theorems \ref{th:A1} and \ref{th:A2} may be unified with an additional parameter introduced.

\begin{theorem}\label{th:And-tri-ext}
	Let $A(n,m,\ell)$ count the number of overpartitions $\lambda$ of $n$ in $\mA_{\{\overline{1}\}}^\veebar$ such that $\sharp_{1,2}(\lambda)+2\sharp_{0,4}(\lambda)=m$ and $\sharp_{2,4}(\lambda)+\sfO(\lambda)=\ell$. Further, let $B(n,m,\ell)$ count the number of $4$-regular partitions into distinct parts with $m$ odd parts and $\ell$ even parts. Then
	\begin{align}
		A(n,m,\ell) = B(n,m,\ell).
	\end{align}
\end{theorem}

\section{A trivariate identity}

To establish Theorem \ref{th:quadruple-ind}, we require the following trivariate relation, which is of independent interest.

\begin{theorem}
	We have
	\begin{align}\label{eq:tri-single}
		\frac{(-x;q)_\infty (xy;q)_\infty}{(x^2yq^2;q^2)_\infty} = \sum_{n\ge 0} \frac{x^n q^{\binom{n}{2}}(1-x^2y^2q^{4n})(xy;q)_n (y;q^2)_n}{(q;q)_n(x^2yq^2;q^2)_n}.
	\end{align}
\end{theorem}

For its proof, we recall that Andrews introduced in \cite{And1968} a family of $q$-series arising from a certain well-poised basic hypergeometric series:
\begin{align*}
	&H_{k,i}(a_1,a_2,a_3;x,q)\\
	&\quad:= \frac{(\frac{xq}{a_1},\frac{xq}{a_2},\frac{xq}{a_3};q)_\infty}{(xq;q)_\infty}\sum_{n\ge 0}\frac{\big(\frac{x^k}{a_1a_2a_3}\big)^nq^{(k-1)n^2+(2-i)n}(1-x^iq^{2ni})(x,a_1,a_2,a_3;q)_n}{(1-x)(q,\frac{xq}{a_1},\frac{xq}{a_2},\frac{xq}{a_3};q)_n}.
\end{align*}
From \cite[p.~439, Eq.~(3.7)]{And1968},
\begin{align}\label{eq:And-H11}
	H_{1,1}(a_1,a_2,a_3;x,q) = \frac{(\frac{xq}{a_1a_2},\frac{xq}{a_2a_3},\frac{xq}{a_3a_1};q)_\infty}{(\frac{xq}{a_1a_2a_3};q)_\infty}.
\end{align}
Also, \cite[p.~439, Eq.~(3.4)]{And1968} tells us that the following $q$-difference equation is valid:
\begin{align}\label{eq:And-H12}
	H_{1,2}(a_1,a_2,a_3;xq,q) &= H_{1,1}(a_1,a_2,a_3;x,q)\notag\\
	&\quad + xq(\sigma_1-xq\sigma_3)H_{1,1}(a_1,a_2,a_3;xq,q),
\end{align}
where $\sigma_j=\sigma_j(a_1^{-1},a_2^{-1},a_3^{-1})$ is the $j$-th elementary symmetric function of $a_1^{-1}$, $a_2^{-1}$ and $a_3^{-1}$.

\begin{proof}
	Define
	\begin{align*}
		h(a_1,a_2;x,q):=\lim_{a_3\to \infty} H_{1,2}(a_1,a_2,a_3;x,q).
	\end{align*}
	Then
	\begin{align*}
		h(a_1,a_2;x,q)&=\frac{(\frac{xq}{a_1},\frac{xq}{a_2};q)_\infty}{(xq;q)_\infty}\sum_{n\ge 0}\frac{\big(\frac{x}{a_1a_2}\big)^n(-1)^nq^{\binom{n}{2}}(1-x^2q^{4n})(x,a_1,a_2;q)_n}{(1-x)(q,\frac{xq}{a_1},\frac{xq}{a_2};q)_n}.
	\end{align*}
	Now taking $(x,a_1,a_2)\mapsto (xy,y^{1/2},-y^{1/2})$ gives
	\begin{align*}
		\sum_{n\ge 0} \frac{x^n q^{\binom{n}{2}}(1-x^2y^2q^{4n})(xy;q)_n (y;q^2)_n}{(q;q)_n(x^2yq^2;q^2)_n} = \frac{(xy;q)_\infty}{(x^2yq^2;q^2)_\infty} h(y^{1/2},-y^{1/2};xy,q).
	\end{align*}
	Meanwhile, it is known from \eqref{eq:And-H11} and \eqref{eq:And-H12} that
	\begin{align*}
		H_{1,2}(a_1,a_2,a_3;x,q)&=\frac{(\frac{x}{a_1a_2},\frac{x}{a_2a_3},\frac{x}{a_3a_1};q)_\infty}{(\frac{x}{a_1a_2a_3};q)_\infty}\\
		&\quad+x\left(\frac{1}{a_1}+\frac{1}{a_2}+\frac{1}{a_3}-\frac{x}{a_1a_2a_3}\right)\frac{(\frac{xq}{a_1a_2},\frac{xq}{a_2a_3},\frac{xq}{a_3a_1};q)_\infty}{(\frac{xq}{a_1a_2a_3};q)_\infty}.
	\end{align*}
	We still let $a_3\to\infty$ and then take $(x,a_1,a_2)\mapsto (xy,y^{1/2},-y^{1/2})$. Thus,
	\begin{align*}
		h(y^{1/2},-y^{1/2};xy,q) = (-x;q)_\infty.
	\end{align*}
	Substituting the above into the previous relation gives the required identity.
\end{proof}

\section{Proof of Theorem \ref{th:quadruple-ind}}

We start with a list of well-known relations for basic hypergeometric series.
\begin{itemize}[leftmargin=*,align=left]
	\renewcommand{\labelitemi}{$\triangleright$}
	\item Euler's first sum \cite[Eq.~(2.2.5)]{And1976}:
	\begin{align}\label{eq:Euler-1}
		\sum_{n\ge 0}\frac{z^n}{(q;q)_n} = \frac{1}{(z;q)_\infty}.
	\end{align}

	\item Euler's second sum \cite[Eq.~(2.2.6)]{And1976}:
	\begin{align}\label{eq:Euler-2}
		\sum_{n\ge 0}\frac{z^nq^{\binom{n}{2}}}{(q;q)_n} = (-z;q)_\infty.
	\end{align}
	
	\item The $q$-binomial theorem \cite[Eq.~(2.2.1)]{And1976}:
	\begin{align}\label{eq:q-Bin-Ser}
		\sum_{n\ge 0}\frac{(a;q)_n z^n}{(q;q)_n} = \frac{(az;q)_\infty}{(z;q)_\infty}.
	\end{align}
\end{itemize}

Now let us establish an equivalent identity of \eqref{eq:quadruple-ind}.

\begin{theorem}\label{th:quadruple-ind-new}
	We have
	\begin{align}\label{eq:quadruple-ind-new}
		&\frac{1}{(xq;q^2)_\infty (yq^2;q^4)_\infty}\notag\\
		&\quad= \sum_{n_1,n_2,n_3,n_4\ge 0}\frac{x^{n_1+n_2+2n_4}y^{n_2+n_3}q^{n_1+3n_2+2n_3+2n_4}(1+x^2yq^{4+4(n_1+n_3)})}{(q^2;q^2)_{n_1}(q^2;q^2)_{n_2}(q^4;q^4)_{n_3}(q^4;q^4)_{n_4}}\notag\\
		&\quad\quad\times q^{2\binom{n_1}{2}+2n_1n_2+4n_1n_3+4n_3n_4}.
	\end{align}
\end{theorem}

\begin{proof}
	We first consider the inner summation over $n_4$ and obtain by \eqref{eq:Euler-1} that
	\begin{align*}
		&\sum_{n_1,n_2,n_3,n_4\ge 0}\frac{x^{n_1+n_2+2n_4}y^{n_2+n_3}q^{2\binom{n_1}{2}+2n_1n_2+4n_1n_3+4n_3n_4+n_1+3n_2+2n_3+2n_4}}{(q^2;q^2)_{n_1}(q^2;q^2)_{n_2}(q^4;q^4)_{n_3}(q^4;q^4)_{n_4}}\\
		&=\sum_{n_1,n_2,n_3\ge 0}\frac{x^{n_1+n_2}y^{n_2+n_3}q^{2\binom{n_1}{2}+2n_1n_2+4n_1n_3+n_1+3n_2+2n_3}}{(q^2;q^2)_{n_1}(q^2;q^2)_{n_2}(q^4;q^4)_{n_3}}\sum_{n_4\ge 0}\frac{(x^2q^{4n_3+2})^{n_4}}{(q^4;q^4)_{n_4}}\\
		&=\sum_{n_1,n_2,n_3\ge 0}\frac{x^{n_1+n_2}y^{n_2+n_3}q^{2\binom{n_1}{2}+2n_1n_2+4n_1n_3+n_1+3n_2+2n_3}}{(q^2;q^2)_{n_1}(q^2;q^2)_{n_2}(q^4;q^4)_{n_3}(x^2q^{4n_3+2};q^4)_\infty}\\
		&=\frac{1}{(x^2q^2;q^4)_\infty}\sum_{n_1,n_2,n_3\ge 0}\frac{x^{n_1+n_2}y^{n_2+n_3}q^{2\binom{n_1}{2}+2n_1n_2+4n_1n_3+n_1+3n_2+2n_3}(x^2q^2;q^4)_{n_3}}{(q^2;q^2)_{n_1}(q^2;q^2)_{n_2}(q^4;q^4)_{n_3}}.
	\end{align*}
	Now we further work on the inner summations over $n_2$ and $n_3$, respectively, with the application of \eqref{eq:Euler-1} and \eqref{eq:q-Bin-Ser}, and find that
	\begin{align*}
		&\sum_{n_1,n_2,n_3,n_4\ge 0}\frac{x^{n_1+n_2+2n_4}y^{n_2+n_3}q^{2\binom{n_1}{2}+2n_1n_2+4n_1n_3+4n_3n_4+n_1+3n_2+2n_3+2n_4}}{(q^2;q^2)_{n_1}(q^2;q^2)_{n_2}(q^4;q^4)_{n_3}(q^4;q^4)_{n_4}}\\
		&=\frac{1}{(x^2q^2;q^4)_\infty}\sum_{n_1\ge 0}\frac{x^{n_1}q^{2\binom{n_1}{2}+n_1}}{(q^2;q^2)_{n_1}}\sum_{n_2\ge 0}\frac{(xyq^{2n_1+3})^{n_2}}{(q^2;q^2)_{n_2}}\sum_{n_3\ge 0}\frac{(yq^{4n_1+2})^{n_3}(x^2q^2;q^4)_{n_3}}{(q^4;q^4)_{n_3}}\\
		&=\frac{1}{(x^2q^2;q^4)_\infty}\sum_{n_1\ge 0}\frac{x^{n_1}q^{2\binom{n_1}{2}+n_1}(x^2yq^{4n_1+4};q^4)_\infty}{(q^2;q^2)_{n_1}(xyq^{2n_1+3};q^2)_\infty(yq^{4n_1+2};q^4)_\infty}\\
		&=\frac{(x^2yq^{4};q^4)_\infty}{(x^2q^2;q^4)_\infty(xyq^{3};q^2)_\infty(yq^{2};q^4)_\infty}\sum_{n_1\ge 0}\frac{x^{n_1}q^{2\binom{n_1}{2}+n_1}(xyq^{3};q^2)_{n_1}(yq^{2};q^4)_{n_1}}{(q^2;q^2)_{n_1}(x^2yq^{4};q^4)_{n_1}}.
	\end{align*}
	Similarly,
	\begin{align*}
		&\sum_{n_1,n_2,n_3,n_4\ge 0}\frac{x^{n_1+n_2+2n_4+2}y^{n_2+n_3+1}q^{2\binom{n_1}{2}+2n_1n_2+4n_1n_3+4n_3n_4+5n_1+3n_2+6n_3+2n_4+4}}{(q^2;q^2)_{n_1}(q^2;q^2)_{n_2}(q^4;q^4)_{n_3}(q^4;q^4)_{n_4}}\\
		&=\frac{x^2yq^4(x^2yq^{8};q^4)_\infty}{(x^2q^2;q^4)_\infty(xyq^{3};q^2)_\infty(yq^{6};q^4)_\infty}\sum_{n_1\ge 0}\frac{x^{n_1}q^{2\binom{n_1}{2}+5n_1}(xyq^{3};q^2)_{n_1}(yq^{6};q^4)_{n_1}}{(q^2;q^2)_{n_1}(x^2yq^{8};q^4)_{n_1}}.
	\end{align*}
	Thus,
	\begin{align*}
		&\RHS\eqref{eq:quadruple-ind-new}\\
		&=\frac{(x^2yq^{4};q^4)_\infty}{(x^2q^2;q^4)_\infty(xyq^{3};q^2)_\infty(yq^{2};q^4)_\infty}\\
		&\quad\times\sum_{n_1\ge 0}\frac{x^{n_1}q^{2\binom{n_1}{2}+n_1}(xyq^{3};q^2)_{n_1}(yq^{2};q^4)_{n_1}}{(q^2;q^2)_{n_1}(x^2yq^{4};q^4)_{n_1}}\left(1+\frac{x^2yq^{4n_1+4}(1-yq^{4n_1+2})}{1-x^2yq^{4n_1+4}}\right)\\
		&=\frac{(x^2yq^{4};q^4)_\infty}{(x^2q^2;q^4)_\infty(xyq^{3};q^2)_\infty(yq^{2};q^4)_\infty}\\
		&\quad\times\sum_{n_1\ge 0}\frac{x^{n_1}q^{2\binom{n_1}{2}+n_1}(xyq^{3};q^2)_{n_1}(yq^{2};q^4)_{n_1}}{(q^2;q^2)_{n_1}(x^2yq^{4};q^4)_{n_1}}\cdot \frac{1-x^2y^2q^{8n_1+6}}{1-x^2yq^{4n_1+4}}\\
		&=\frac{(x^2yq^{8};q^4)_\infty}{(x^2q^2;q^4)_\infty(xyq^{3};q^2)_\infty(yq^{2};q^4)_\infty}\\
		&\quad\times\sum_{n_1\ge 0}\frac{x^{n_1}q^{2\binom{n_1}{2}+n_1}(1-x^2y^2q^{8n_1+6})(xyq^{3};q^2)_{n_1}(yq^{2};q^4)_{n_1}}{(q^2;q^2)_{n_1}(x^2yq^{8};q^4)_{n_1}}.
	\end{align*}
	Finally, in \eqref{eq:tri-single}, we take $(x,y,q)\mapsto (xq,yq^2,q^2)$. Then
	\begin{align*}
		\RHS\eqref{eq:quadruple-ind-new}&= \frac{(x^2yq^{8};q^4)_\infty}{(x^2q^2;q^4)_\infty(xyq^{3};q^2)_\infty(yq^{2};q^4)_\infty}\cdot \frac{(-xq;q^2)_\infty (xyq^3;q^2)_\infty}{(x^2yq^8;q^4)_\infty}\\
		&=\frac{1}{(xq;q^2)_\infty (yq^2;q^4)_\infty},
	\end{align*}
	as required.
\end{proof}

To see why \eqref{eq:quadruple-ind-new} and \eqref{eq:quadruple-ind} are equivalent, we need a functional operator $\cB$ defined on $\mathbb{C}[[q]][[x,y]]$ by
\begin{align*}
	\cB\left(\sum_{m,n\ge 0}c_{m,n}x^m y^n\right):= \sum_{m,n\ge 0}c_{m,n}q^{2\binom{m}{2}+4\binom{n}{2}} x^m y^n,
\end{align*}
where the coefficients $c_{m,n}$ are in $\mathbb{C}[[q]]$. This operator can be treated as a special case of the $q$-Borel operators \cite{Che2022b,Ram1992,Zha2002}.

\begin{proof}[Proof of \eqref{eq:quadruple-ind} from \eqref{eq:quadruple-ind-new}]
	Note that
	\begin{align*}
		&\cB\big(\RHS\eqref{eq:quadruple-ind-new}\big)\\
		&=\sum_{n_1,n_2,n_3,n_4\ge 0}\frac{x^{n_1+n_2+2n_4}y^{n_2+n_3}q^{2\binom{n_1}{2}+2n_1n_2+4n_1n_3+4n_3n_4+n_1+3n_2+2n_3+2n_4}}{(q^2;q^2)_{n_1}(q^2;q^2)_{n_2}(q^4;q^4)_{n_3}(q^4;q^4)_{n_4}}\notag\\
		&\quad\times \big(q^{2\binom{n_1+n_2+2n_4}{2}+4\binom{n_2+n_3}{2}}+x^2yq^{2\binom{n_1+n_2+2n_4+2}{2}+4\binom{n_2+n_3+1}{2}+4+4(n_1+n_3)}\big)\\
		&= \sum_{n_1,n_2,n_3,n_4\ge 0}\frac{x^{n_1+n_2+2n_4}y^{n_2+n_3}q^{n_1+3n_2+2n_3+4n_4}(1+x^2yq^{6+8(n_1+n_2+n_3+n_4)})}{(q^2;q^2)_{n_1}(q^2;q^2)_{n_2}(q^4;q^4)_{n_3}(q^4;q^4)_{n_4}}\notag\\
		&\quad\times q^{4\binom{n_1}{2}+6\binom{n_2}{2}+4\binom{n_3}{2}+8\binom{n_4}{2}+4n_1n_2+4n_1n_3+4n_1n_4+4n_2n_3+4n_2n_4+4n_3n_4},
	\end{align*}
	which is exactly the right-hand side of \eqref{eq:quadruple-ind}. On the other hand, we rewrite the left-hand side of \eqref{eq:quadruple-ind-new} in light of \eqref{eq:Euler-1},
	\begin{align*}
		\LHS\eqref{eq:quadruple-ind-new} = \sum_{m_1,m_2\ge 0}\frac{x^{m_1}y^{m_2}q^{m_1+2m_2}}{(q^2;q^2)_{m_1}(q^4;q^4)_{m_2}}.
	\end{align*}
	Hence,
	\begin{align*}
		\cB\big(\LHS\eqref{eq:quadruple-ind-new}\big) &= \sum_{m_1,m_2\ge 0}\frac{x^{m_1}y^{m_2}q^{2\binom{m_1}{2}+4\binom{m_2}{2}+m_1+2m_2}}{(q^2;q^2)_{m_1}(q^4;q^4)_{m_2}}\\
		&= (-xq;q^2)_\infty (-yq^2;q^4)_\infty,
	\end{align*}
	where \eqref{eq:Euler-2} is applied. Finally,
	\begin{align*}
		\RHS\eqref{eq:quadruple-ind} = \cB\big(\RHS\eqref{eq:quadruple-ind-new}\big) = \cB\big(\LHS\eqref{eq:quadruple-ind-new}\big) = \LHS\eqref{eq:quadruple-ind},
	\end{align*}
	as desired.
\end{proof}

\section{Span one linked partition ideals}

Now we shall consider the generating functions related to the overpartitions in $\mA$. For this purpose, we take advantage of the framework of \emph{span one linked partition ideals} introduced by Andrews \cite{And1972,And1974b,And1975} in the 1970s and reflourished in a series of recent projects mainly led by Chern \cite{ACL2022,Che2020,Che2022a,Che2022b,CL2020}. It is necessary to point out that linked partition ideals are originally considered over ordinary partitions; see, for instance, \cite[Chapter 8]{And1976} or \cite[Definition 2.1]{ACL2022}. However, according to the generic setting introduced in \cite{Che2020}, including \textbf{overpartitions} will \emph{not} bring about any extra issue.

\begin{definition}
	Assume that we are given
	\begin{itemize}[leftmargin=*,align=left]
		\renewcommand{\labelitemi}{\scriptsize$\blacktriangleright$}
		
		\item a finite set $\Pi=\{\pi_1,\pi_2,\ldots,\pi_K\}$ of \textbf{overpartitions} with $\pi_1=\varnothing$, the empty partition,
		
		\item a \textit{map of linking sets}, $\cL:\Pi\to P(\Pi)$, the power set of $\Pi$, with especially, $\cL(\pi_1)=\cL(\varnothing)=\Pi$ and $\pi_1=\varnothing\in \cL(\pi_k)$ for any $1\le k\le K$,
		
		\item and a positive integer $T$, called the \textit{modulus}, which is greater than or equal to the largest part among all \textbf{overpartitions} in $\Pi$.
	\end{itemize}
	We say a \textit{span one linked partition ideal} $\mI=\mI(\langle\Pi,\cL\rangle,T)$ is the collection of all \textbf{overpartitions} of the form
	\begin{align}\label{eq:decomp}
		\lambda&=\phi^0(\lambda_0)\oplus \phi^T(\lambda_1)\oplus \cdots \oplus \phi^{NT}(\lambda_N)\oplus \phi^{(N+1)T}(\pi_1)\oplus \phi^{(N+2)T}(\pi_1)\oplus \cdots\notag\\
		&=\phi^0(\lambda_0)\oplus \phi^T(\lambda_1)\oplus \cdots \oplus \phi^{NT}(\lambda_N),
	\end{align}
	where $\lambda_i\in\cL(\lambda_{i-1})$ for each $i$ and $\lambda_N$ is not the empty partition. We also include in $\mI$ the empty partition, which corresponds to $\phi^{0}(\pi_1)\oplus \phi^{T}(\pi_1)\oplus \cdots$. Here for any two \textbf{overpartitions} $\mu$ and $\nu$, $\mu\oplus\nu$ gives an \textbf{overpartition} by collecting all parts in $\mu$ and $\nu$, and $\phi^m(\mu)$ gives an \textbf{overpartition} by adding $m$ to each part of $\mu$ with \textbf{overlines preserved}.
\end{definition}

Recall that $\mA$ denotes the set of overpartitions such that
\begin{enumerate}[label={\textup{(\arabic*)~}},leftmargin=*,labelsep=0cm,align=left]
	\item Only odd parts may be overlined;
	
	\item The difference between any two parts is $\ge 4$ and the inequality is strict if the larger one is overlined or divisible by $4$.
\end{enumerate}

\begin{lemma}
	$\mA$ equals the span one linked partition ideal $\mI(\langle\Pi,\cL\rangle,4)$, where
	$\Pi=\{\pi_1=\varnothing,\pi_2=(1),\pi_3=(\overline{1}),\pi_4=(2),\pi_5=(3),\pi_6=(\overline{3}),\pi_7=(4)\}$ and
	\begin{equation*}
		\left\{
		\begin{aligned}
			\cL(\pi_1)&=\{\pi_1,\pi_2,\pi_3,\pi_4,\pi_5,\pi_6,\pi_7\},\\
			\cL(\pi_2)=\cL(\pi_3)&=\{\pi_1,\pi_2,\pi_4,\pi_5,\pi_6,\pi_7\},\\
			\cL(\pi_4)&=\{\pi_1,\pi_4,\pi_5,\pi_6,\pi_7\},\\
			\cL(\pi_5)=\cL(\pi_6)&=\{\pi_1,\pi_5,\pi_7\},\\
			\cL(\pi_7)&=\{\pi_1\}.
		\end{aligned}
		\right.
	\end{equation*}
\end{lemma}

\begin{proof}
	It is clear that all overpartitions in $\mI(\langle\Pi,\cL\rangle,4)$ satisfy the conditions for $\mA$. For the other direction, we decompose each overpartition in $\mA$ into blocks $\bB_0,\bB_1,\ldots$ such that all parts (including those that are overlined) between $4i+1$ and $4i+4$ fall into block $\bB_i$. It is plain that $\phi^{-4i}(\bB_i)$ is exclusively from $\Pi$. Further, if $\phi^{-4i}(\bB_i)$ is $\pi_1$ so that $\bB_i$ is $\varnothing$, then $\phi^{-4(i+1)}(\bB_{i+1})$ can be any among $\Pi$. If $\phi^{-4i}(\bB_i)$ is $\pi_2$ or $\pi_3$ so that $\bB_i$ is $(4i+1)$ or $(\overline{4i+1})$, then $\bB_{i+1}$ cannot be $(\overline{4i+5})$ by the second condition for $\mA$ so that $\phi^{-4(i+1)}(\bB_{i+1})$ cannot be $\pi_3$. One may carry out similar arguments for other possibilities of $\phi^{-4i}(\bB_i)$ and the details are omitted.
\end{proof}

\begin{example}
	As in \eqref{eq:decomp}, we decompose the overpartition $\overline{1}+8+14+\overline{19}+23+27$ by
	$$\phi^{0}(\overline{1})\oplus \phi^{4}(4)\oplus \phi^{8}(\varnothing) \oplus \phi^{12}(2) \oplus \phi^{16}(\overline{3}) \oplus \phi^{20}(3) \oplus \phi^{24}(3),$$
	which corresponds to the chain $\pi_3\pi_7\pi_1\pi_4\pi_6\pi_5\pi_5\pi_1\pi_1\cdots$.
\end{example}

Throughout, we always decompose overpartitions $\lambda\in \mA=\mI(\langle\Pi,\cL\rangle,4)$ as in \eqref{eq:decomp}. Now define for $1\le k\le 7$:
\begin{align}
	G_k(x)=G_k(x,y_1,y_2,z,q):=\sum_{\substack{\lambda\in\mA\\\lambda_0=\pi_k}}x^{\sharp(\lambda)}y_1^{\sharp_{2,4}(\lambda)}y_2^{\sharp_{0,4}(\lambda)}z^{\sfO(\lambda)}q^{|\lambda|}.
\end{align}
In other words, $G_k(x)$ is the generating function for overpartitions in $\mA$ whose first decomposed block $\bB_0$ equals $\pi_k$. From the above construction, it is plain that
\begin{align*}
	G_k(x)=x^{\sharp(\pi_k)}y_1^{\sharp_{2,4}(\pi_k)}y_2^{\sharp_{0,4}(\pi_k)}z^{\sfO(\pi_k)}q^{|\pi_k|}\sum_{j:\pi_j\in\cL(\pi_k)} G_j(xq^4).
\end{align*}
Hence,
\begin{equation}
	\begin{pmatrix}
		G_1(x)\\
		G_2(x)\\
		\vdots\\
		G_{7}(x)
	\end{pmatrix}
	=
	\bW.\bA.
	\begin{pmatrix}
		G_1(xq^4)\\
		G_2(xq^4)\\
		\vdots\\
		G_{7}(xq^4)
	\end{pmatrix},
\end{equation}
where
$$\bW=\diag(1,xq,xzq,xy_1q^2,xq^3,xzq^3,xy_2q^4)$$
and
$$\setcounter{MaxMatrixCols}{15}
\bA=\begin{pmatrix}
	1 & 1 & 1 & 1 & 1 & 1 & 1\\
	1 & 1 & 0 & 1 & 1 & 1 & 1\\
	1 & 1 & 0 & 1 & 1 & 1 & 1\\
	1 & 0 & 0 & 1 & 1 & 1 & 1\\
	1 & 0 & 0 & 0 & 1 & 0 & 1\\
	1 & 0 & 0 & 0 & 1 & 0 & 1\\
	1 & 0 & 0 & 0 & 0 & 0 & 0
\end{pmatrix}.$$
We further write
\begin{align}\label{eq:F-G}
	\begin{pmatrix}
		F_1(x)\\
		F_2(x)\\
		\vdots\\
		F_{7}(x)
	\end{pmatrix}=\bA.\begin{pmatrix}
		G_1(x)\\
		G_2(x)\\
		\vdots\\
		G_{7}(x)
	\end{pmatrix}.
\end{align}
Then
\begin{align}\label{eq:F}
	\begin{pmatrix}
		F_1(x)\\
		F_2(x)\\
		\vdots\\
		F_{7}(x)
	\end{pmatrix}=\bA.\bW.\begin{pmatrix}
		F_1(xq^4)\\
		F_2(xq^4)\\
		\vdots\\
		F_{7}(xq^4)
	\end{pmatrix}.
\end{align}

\section{Quinvariate generating functions}

Here our object is to establish related generating functions for $\mA$. Letting $S$ be a collection of parts, we denote by $\mA_S$ the subset of overpartitions in $\mA$ such that parts from $S$ are forbidden.

\begin{theorem}\label{th:quin-gf}
	We have
	\begin{align}
		&\sum_{\lambda\in \mA} x^{\sharp(\lambda)}y_1^{\sharp_{2,4}(\lambda)}y_2^{\sharp_{0,4}(\lambda)}z^{\sfO(\lambda)}q^{|\lambda|}\notag\\
		&\qquad = \sum_{n_1,n_2,n_3,n_4\ge 0}\frac{x^{n_1+n_2+n_3+n_4}y_1^{n_3}y_2^{n_4}z^{n_2}q^{n_1+n_2+2n_3+4n_4}}{(q^2;q^2)_{n_1}(q^2;q^2)_{n_2}(q^4;q^4)_{n_3}(q^4;q^4)_{n_4}}\notag\\
		&\qquad\quad\times q^{4\binom{n_1}{2}+6\binom{n_2}{2}+4\binom{n_3}{2}+8\binom{n_4}{2}+4n_1n_2+4n_1n_3+4n_1n_4+4n_2n_3+4n_2n_4+4n_3n_4},\\
		&\!\!\sum_{\lambda\in \mA_{\{\overline{1}\}}} x^{\sharp(\lambda)}y_1^{\sharp_{2,4}(\lambda)}y_2^{\sharp_{0,4}(\lambda)}z^{\sfO(\lambda)}q^{|\lambda|}\notag\\
		&\qquad = \sum_{n_1,n_2,n_3,n_4\ge 0}\frac{x^{n_1+n_2+n_3+n_4}y_1^{n_3}y_2^{n_4}z^{n_2}q^{n_1+3n_2+2n_3+4n_4}}{(q^2;q^2)_{n_1}(q^2;q^2)_{n_2}(q^4;q^4)_{n_3}(q^4;q^4)_{n_4}}\notag\\
		&\qquad\quad\times q^{4\binom{n_1}{2}+6\binom{n_2}{2}+4\binom{n_3}{2}+8\binom{n_4}{2}+4n_1n_2+4n_1n_3+4n_1n_4+4n_2n_3+4n_2n_4+4n_3n_4},\label{eq:gf-A-1ov}\\
		&\!\!\!\!\sum_{\lambda\in \mA_{\{1,\overline{1}\}}} x^{\sharp(\lambda)}y_1^{\sharp_{2,4}(\lambda)}y_2^{\sharp_{0,4}(\lambda)}z^{\sfO(\lambda)}q^{|\lambda|}\notag\\
		&\qquad = \sum_{n_1,n_2,n_3,n_4\ge 0}\frac{x^{n_1+n_2+n_3+n_4}y_1^{n_3}y_2^{n_4}z^{n_2}q^{3n_1+3n_2+2n_3+4n_4}}{(q^2;q^2)_{n_1}(q^2;q^2)_{n_2}(q^4;q^4)_{n_3}(q^4;q^4)_{n_4}}\notag\\
		&\qquad\quad\times q^{4\binom{n_1}{2}+6\binom{n_2}{2}+4\binom{n_3}{2}+8\binom{n_4}{2}+4n_1n_2+4n_1n_3+4n_1n_4+4n_2n_3+4n_2n_4+4n_3n_4},\\
		&\!\!\!\!\!\!\!\!\sum_{\lambda\in \mA_{\{1,\overline{1},2,\overline{3}\}}} x^{\sharp(\lambda)}y_1^{\sharp_{2,4}(\lambda)}y_2^{\sharp_{0,4}(\lambda)}z^{\sfO(\lambda)}q^{|\lambda|}\notag\\
		&\qquad = \sum_{n_1,n_2,n_3,n_4\ge 0}\frac{x^{n_1+n_2+n_3+n_4}y_1^{n_3}y_2^{n_4}z^{n_2}q^{3n_1+5n_2+6n_3+4n_4}}{(q^2;q^2)_{n_1}(q^2;q^2)_{n_2}(q^4;q^4)_{n_3}(q^4;q^4)_{n_4}}\notag\\
		&\qquad\quad\times q^{4\binom{n_1}{2}+6\binom{n_2}{2}+4\binom{n_3}{2}+8\binom{n_4}{2}+4n_1n_2+4n_1n_3+4n_1n_4+4n_2n_3+4n_2n_4+4n_3n_4}.
	\end{align}
\end{theorem}

To begin with, we note that
\begin{align*}
	\sum_{\lambda\in \mA} x^{\sharp(\lambda)}y_1^{\sharp_{2,4}(\lambda)}y_2^{\sharp_{0,4}(\lambda)}z^{\sfO(\lambda)}q^{|\lambda|}&=\sum_{k\in\{1,2,3,4,5,6,7\}}G_k(x)=F_1(x),\\
	\sum_{\lambda\in \mA_{\{\overline{1}\}}} x^{\sharp(\lambda)}y_1^{\sharp_{2,4}(\lambda)}y_2^{\sharp_{0,4}(\lambda)}z^{\sfO(\lambda)}q^{|\lambda|}&=\sum_{k\in\{1,2,4,5,6,7\}}G_k(x)=F_2(x)=F_3(x),\\
	\sum_{\lambda\in \mA_{\{1,\overline{1}\}}} x^{\sharp(\lambda)}y_1^{\sharp_{2,4}(\lambda)}y_2^{\sharp_{0,4}(\lambda)}z^{\sfO(\lambda)}q^{|\lambda|}&=\sum_{k\in\{1,4,5,6,7\}}G_k(x)=F_4(x),\\
	\sum_{\lambda\in \mA_{\{1,\overline{1},2,\overline{3}\}}} x^{\sharp(\lambda)}y_1^{\sharp_{2,4}(\lambda)}y_2^{\sharp_{0,4}(\lambda)}z^{\sfO(\lambda)}q^{|\lambda|}&=\sum_{k\in\{1,5,7\}}G_k(x)=F_5(x)=F_6(x).
\end{align*}
Thus it suffices to determine the expression of each $F_k(x)$. If we treat \eqref{eq:F} as a system of $q$-difference equations, its formal power series solution $\big(F_1(x),\ldots,F_7(x)\big)$ is uniquely determined by $\big(F_1(0),\ldots,F_7(0)\big)$. Further, according to our construction, $F_k(0)=1$ for each $k$.

Recall that in \cite{Che2020} and \cite{Che2022a}, a generic family of $q$-multi-summations was considered. Let $R$ be a given positive integer and fix a symmetric matrix $\underline{\boldsymbol{\alpha}}=(\alpha_{i,j})\in\Mat_{R\times R}(\mathbb{N})$ and a vector $\underline{\mathbf{A}}=(A_r)\in \mathbb{N}_{>0}^R$. Also fix $J$ vectors $\underline{\boldsymbol{\gamma_j}}=(\gamma_{j,r})\in \mathbb{N}_{\ge 0}^R$ for $j=1,2,\ldots,J$. Define for indeterminates $x_1,x_2,\ldots,x_J$ and $q$ the following $q$-multi-summation $H(\ubb)=H(\beta_1,\ldots,\beta_R)$ with $\ubb\in\mathbb{Z}^R$:
\begin{align*}
	H(\ubb)=H(\beta_1,\ldots,\beta_R)&:=\sum_{n_1,\ldots,n_R\ge 0}\frac{x_1^{\sum_{r=1}^R \gamma_{1,r} n_r}\cdots x_J^{\sum_{r=1}^R \gamma_{J,r} n_r}}{(q^{A_1};q^{A_1})_{n_1}\cdots (q^{A_R};q^{A_R})_{n_R}}\notag\\
	&\;\quad\times q^{\sum_{r=1}^R \alpha_{r,r}\binom{n_r}{2}+\sum_{1\le i< j\le R}\alpha_{i,j}n_i n_j+ \sum_{r=1}^R \beta_r n_r}.
\end{align*}
We require a recurrence for $H(\ubb)$ given in \cite[Lemma 2.1]{Che2022a}.

\begin{lemma}\label{le:rec-key}
	For $1\le r\le R$, we have
	\begin{align*}
		H(\beta_1,\ldots,\beta_r,\ldots,\beta_R)&=H(\beta_1,\ldots,\beta_r+A_r,\ldots,\beta_R)\notag\\
		&+x_1^{\gamma_{1,r}}\cdots x_J^{\gamma_{J,r}}q^{\beta_r}H(\beta_1+\alpha_{r,1},\ldots,\beta_r+\alpha_{r,r},\ldots,\beta_R+\alpha_{r,R}).
	\end{align*}
\end{lemma}

As in \cite{Che2022a}, we illustrate the above relation by a binary tree with the coordinate $\beta_r$ shown in boldface; see Figure \ref{fig:node-children}.

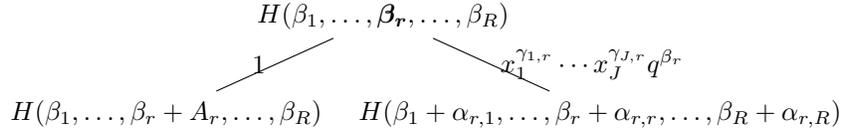
\begin{figure}[h]
	\caption{Node $H(\beta_1,\ldots,\beta_r,\ldots,\beta_R)$ and its children}\label{fig:node-children}
	\medskip\medskip
	\begin{forest}
		for tree={l sep=20pt}
		[{$H(\beta_1,\ldots,\boldsymbol{\beta_r},\ldots,\beta_R)$} 
		[{$H(\beta_1,\ldots,\beta_r+A_r,\ldots,\beta_R)$}, edge label={node[midway,left] {$1$}} ]
		[{$H(\beta_1+\alpha_{r,1},\ldots,\beta_r+\alpha_{r,r},\ldots,\beta_R+\alpha_{r,R})$}, edge label={node[midway,right] {$x_1^{\gamma_{1,r}}\cdots x_J^{\gamma_{J,r}}q^{\beta_r}$}} ] 
		]
	\end{forest}
\end{figure}

Now let us choose
$$\uba=\begin{pmatrix}
	4 & 4 & 4 & 4\\4 & 6 & 4 & 4\\4 & 4 & 4 & 4\\4 & 4 & 4 & 8
\end{pmatrix},\qquad\qquad\qquad \ubA=(2,2,4,4),$$
and
\begin{alignat*}{2}
	x_1&=x,\qquad\qquad &&\ubgg{1}=(1,1,1,1),\\
	x_2&=y_1,\qquad\qquad &&\ubgg{2}=(0,0,1,0),\\
	x_3&=y_2,\qquad\qquad &&\ubgg{3}=(0,0,0,1),\\
	x_4&=z,\qquad\qquad &&\ubgg{4}=(0,1,0,0).
\end{alignat*}
To prove Theorem \ref{th:quin-gf}, it is sufficient to confirm that
\begin{align}\label{eq:H-relation-main}
	\scalebox{.685}{%
	$\begin{pmatrix}
		H(1,1,2,4)\\
		H(1,3,2,4)\\
		H(1,3,2,4)\\
		H(3,3,2,4)\\
		H(3,5,6,4)\\
		H(3,5,6,4)\\
		H(5,5,6,8)
	\end{pmatrix}=\begin{pmatrix}
	1 & 1 & 1 & 1 & 1 & 1 & 1\\
	1 & 1 & 0 & 1 & 1 & 1 & 1\\
	1 & 1 & 0 & 1 & 1 & 1 & 1\\
	1 & 0 & 0 & 1 & 1 & 1 & 1\\
	1 & 0 & 0 & 0 & 1 & 0 & 1\\
	1 & 0 & 0 & 0 & 1 & 0 & 1\\
	1 & 0 & 0 & 0 & 0 & 0 & 0
	\end{pmatrix}.\begin{pmatrix}
	1\\
	& xq\\
	&& xzq\\
	&&& xy_1q^2\\
	&&&& xq^3\\
	&&&&& xzq^3\\
	&&&&&& xy_2q^4
	\end{pmatrix}.\begin{pmatrix}
	H(5,5,6,8)\\
	H(5,7,6,8)\\
	H(5,7,6,8)\\
	H(7,7,6,8)\\
	H(7,9,10,8)\\
	H(7,9,10,8)\\
	H(9,9,10,12)
\end{pmatrix}.$}
\end{align}

\begin{figure}[ht]
	\caption{The binary tree for \eqref{eq:H-relation-main}}\label{fig:btree-proof}
	\medskip\medskip
	{\footnotesize
		\begin{forest}
			for tree={l sep=15pt}
			[{$H(1,\boldsymbol{1},2,4)$} 
				[{$H(\boldsymbol{1},3,2,4)$}, edge label={node[midway,left] {$1$}}  
					[{$H(3,3,\boldsymbol{2},4)$}, edge label={node[midway,left] {$1$}} 
						[{$H(3,\boldsymbol{3},6,4)$}, edge label={node[midway,left] {$1$}} 
							[{$H(\boldsymbol{3},5,6,4)$}, edge label={node[midway,left] {$1$}} 
								[{$H(5,5,6,\boldsymbol{4})$}, edge label={node[midway,left] {$1$}} 
									[{$H(5,5,6,8)$}, edge label={node[midway,left] {$1$}} 
									]
									[{$H(9,9,10,12)$}, edge label={node[midway,right] {$xy_2q^{4}$}}
									]
								]
								[{$H(7,9,10,8)$}, edge label={node[midway,right] {$xq^{3}$}}
								]
							]
							[{$H(7,9,10,8)$}, edge label={node[midway,right] {$xzq^{3}$}}
							]
						]
						[{$H(7,7,6,8)$}, edge label={node[midway,right] {$xy_1q^{2}$}}
						]
					] 
					[{$H(5,7,6,8)$}, edge label={node[midway,right] {$xq$}}
					]
				]
				[{$H(5,7,6,8)$}, edge label={node[midway,right] {$xzq$}}
				]
			]
		\end{forest}
	}
\end{figure}
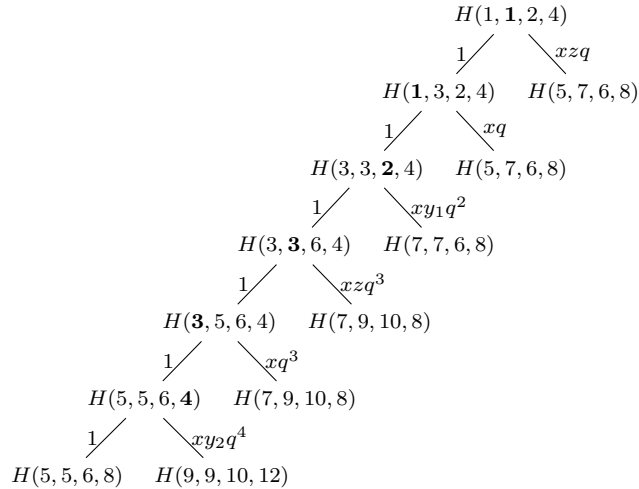

\begin{proof}
	We make use of Lemma \ref{le:rec-key} and illustrate the proof by the binary tree in Figure \ref{fig:btree-proof}. For instance, from the node $H(\boldsymbol{3},5,6,4)$ at the fifth level, we apply Lemma \ref{le:rec-key} to the first coordinate and obtain
	\begin{align*}
		H(\boldsymbol{3},5,6,4) = H(5,5,6,4) + xq^3H(7,9,10,8).
	\end{align*}
	We further apply Lemma \ref{le:rec-key} to the fourth coordinate of $H(5,5,6,\boldsymbol{4})$ and obtain
	\begin{align*}
		H(5,5,6,\boldsymbol{4}) = H(5,5,6,8) + xy_2q^4H(9,9,10,12).
	\end{align*}
	Hence,
	\begin{align*}
		H(3,5,6,4) = H(5,5,6,8) + xq^3H(7,9,10,8) + xy_2q^4H(9,9,10,12),
	\end{align*}
	thereby confirming the fifth and sixth rows of \eqref{eq:H-relation-main}. Other rows can be argued in the same vein.
\end{proof}

\section{Proof of Theorem \ref{th:And-tri-ext}}

Theorem \ref{th:And-tri-ext} is a direct consequence of \eqref{eq:quadruple-ind} and \eqref{eq:gf-A-1ov}. Recall that $\mA_{\{\overline{1}\}}^\veebar$ denotes the set of overpartitions such that
\begin{enumerate}[label={\textup{(\arabic*)~}},leftmargin=*,labelsep=0cm,align=left]
	\item Only odd parts {\bfseries\boldmath larger than $1$} may be overlined;
	
	\item The difference between any two parts is $\ge 4$ and the inequality is strict if the larger one is overlined or divisible by $4$ with {\bfseries\boldmath the exception that $\overline{5}$ and $1$ may simultaneously appear as parts}.
\end{enumerate}

Now there are two cases. \textbf{(i).}~If $\overline{5}$ and $1$ do not simultaneously appear as parts, then such overpartitions are exactly those in $\mA_{\{\overline{1}\}}$. \textbf{(ii).}~If $\overline{5}$ and $1$ simultaneously appear as parts, then apart from them, the smallest part is at least of size $9$, while $\overline{9}$ cannot be a part since if this is the case, we have parts $\overline{9}+\overline{5}$, violating the second condition. Now removing parts $\overline{5}$ and $1$, subtracting $8$ from each of the remaining parts, and preserving all overlines, we again get a partition in $\mA_{\{\overline{1}\}}$. Consequently,
\begin{align*}
	\sum_{\lambda\in \mA_{\{\overline{1}\}}^\veebar}x^{\sharp(\lambda)}y_1^{\sharp_{2,4}(\lambda)}y_2^{\sharp_{0,4}(\lambda)}z^{\sfO(\lambda)}q^{|\lambda|}&= \sum_{\lambda\in \mA_{\{\overline{1}\}}}x^{\sharp(\lambda)}y_1^{\sharp_{2,4}(\lambda)}y_2^{\sharp_{0,4}(\lambda)}z^{\sfO(\lambda)}q^{|\lambda|}\\
	&+x^2zq^6\sum_{\lambda\in \mA_{\{\overline{1}\}}}(xq^8)^{\sharp(\lambda)}y_1^{\sharp_{2,4}(\lambda)}y_2^{\sharp_{0,4}(\lambda)}z^{\sfO(\lambda)}q^{|\lambda|}.
\end{align*}
It follows that
\begin{align*}
	&\sum_{\ell,m,n\ge 0}A(n,m,\ell)x^m y^\ell q^n\\
	&\quad= \sum_{\lambda\in \mA_{\{\overline{1}\}}^\veebar}x^{\sharp(\lambda)}(x^{-1}y)^{\sharp_{2,4}(\lambda)}x^{\sharp_{0,4}(\lambda)}y^{\sfO(\lambda)}q^{|\lambda|}\\
	&\quad= \sum_{n_1,n_2,n_3,n_4\ge 0}\frac{x^{n_1+n_2+2n_4}y^{n_2+n_3}q^{n_1+3n_2+2n_3+4n_4}(1+x^2yq^{6+8(n_1+n_2+n_3+n_4)})}{(q^2;q^2)_{n_1}(q^2;q^2)_{n_2}(q^4;q^4)_{n_3}(q^4;q^4)_{n_4}}\notag\\
	&\quad\quad\times q^{4\binom{n_1}{2}+6\binom{n_2}{2}+4\binom{n_3}{2}+8\binom{n_4}{2}+4n_1n_2+4n_1n_3+4n_1n_4+4n_2n_3+4n_2n_4+4n_3n_4}\\
	&\quad=(-xq;q^2)_\infty (-yq^2;q^4)_\infty\\
	&\quad=\sum_{\ell,m,n\ge 0}B(n,m,\ell)x^m y^\ell q^n.
\end{align*}

\section{Conclusion}

The results in this paper together with those in \cite{ACL2022,Che2020,Che2022a,Che2022b,CL2020} make clear that the power of linked partition ideals, first defined in \cite{And1974b}, is only now coming into prominence. In addition, the study of linked partition ideals began with an effort to expand the world of partition identities via $q$-difference equations. This latter topic, considered extensively in \cite{And1968} and utilized effectively in this paper, should further develop in parallel with the theory of linked partition ideals.  

Finally, we see in this paper a new level of partition identity refinement building on the refinements in \cite{And2022} which in turn refined Glaisher's ancient theorem \cite{Gla1883}. It is natural to ask which of the classical partition identities are amenable to refinements and what are the limits of this exploration. We note, for example, that the Rogers--Ramanujan identities themselves have no known refinements along the lines considered here.

\subsection*{Acknowledgements}

G.~E.~Andrews was supported by a grant (\#633284) from the Simons Foundation. S.~Chern was supported by a Killam Postdoctoral Fellowship from the Killam Trusts.

\bibliographystyle{amsplain}

\end{document}